\documentclass[11pt]{amsart}
\usepackage{graphicx}
\usepackage{amssymb}
\usepackage{amsmath}
\usepackage{amsthm}
\usepackage{mathtools}
\usepackage{tikz}
\usepackage{psfrag}
\usepackage{xcolor}

\newtheorem{thm}{Theorem}[section]
\newtheorem{lem}[thm]{Lemma}

\theoremstyle{definition}

\newtheorem{remk}{Remark}[section]

\newcommand{\N}{\mathbb N}

\newcommand{\R}{\mathbb R}

\makeatletter
\numberwithin{equation}{section}
\makeatother

\begin{document}

% author information

       % first author

       \author{Seonghak Kim}
       \address{Institute for Mathematical Sciences\\ Renmin University of China \\  Beijing 100872, PRC}
       \email{kimseo14@gmail.com}

       % second author

       % current address, usually not needed because it is the same as the
       % regular address

       % title

      % \title[On one-dimensional  forward-backward parabolic equations]{On  one-dimensional forward-backward parabolic  equations  with   linear convection and reaction}
\title[Rate of convergence for 1-D quasilinear problem]{Rate of convergence for one-dimensional quasilinear parabolic problem and its applications}


\begin{abstract}
Based on a comparison principle, we derive an exponential rate of convergence for solutions to the initial-boundary value problem for a class of quasilinear parabolic equations in one space dimension. We then apply the result to some models in population dynamics and image processing.
\end{abstract}
\maketitle

\section{Introduction}
In this paper, we study large time behaviors of solutions to the initial-boundary value problem for a class of quasilinear advection-diffusion equations in one space dimension:
\begin{equation}\label{main-ibP}
\left\{
\begin{array}{ll}
  \rho_t=(\sigma(\rho))_{xx} & \mbox{in $\Omega\times(0,\infty)$,} \\
  \rho=\rho_0 & \mbox{on $\Omega\times\{t=0\}$,} \\
  \rho(0,t)=\rho(L,t)=0 & \mbox{for $t\in(0,\infty)$}.
\end{array}\right.
\end{equation}
Here, $\Omega=(0,L)\subset\R$ is the spatial domain of a given length $L>0$, $\sigma=\sigma(s)\in C^2(\R)$ is the flux function, $\rho_0=\rho_0(x)$ is a given initial datum, and $\rho=\rho(x,t)$ is a solution to the problem.

After the pioneering work of {Zelenjak} \cite{Ze} and {Matano} \cite{Ma}, there have been many studies on the analysis of  $\omega$-limit sets and large time behaviors of solutions to some semilinear or quasilinear parabolic problems in one space dimension; see, e.g., \cite{CM,VZL,GS}. In particular, the work \cite{VZL} considered a general quasilinear problem, which includes problem (\ref{main-ibP}) as a special case, and proved the convergence of a global classical solution to a unique steady state in the space $C^1(\bar\Omega)$ as $t\to\infty$ if the solution itself is assumed to be bounded in $C^1(\bar\Omega)$ uniformly in $t\ge0$.

In the present article, we show that a global classical solution to problem (\ref{main-ibP}) converges uniformly  to the steady state $0$ as $t\to \infty$ at an exponential rate. The difference of our result from that in \cite{VZL} is in two-fold. First, we do not assume the uniform boundedness of the solution to (\ref{main-ibP}) in  $C^1(\bar\Omega)$. Second, we obtain specific exponential rates of convergence depending only on the initial datum $\rho_0$, the flux function $\sigma$ and the size $L$ of the spatial domain $\Omega$.

The main motivation of our result lies in its application to study large time behaviors of global \emph{weak} solutions to some problems modeling aggregative movement in population dynamics \cite{Tu, AS} and enhancement of a noisy picture in image processing \cite{PM}. However, proving the global existence of such solutions is not at all obvious since those problems are often \emph{forward} and \emph{backward parabolic} so that the standard methods of parabolic equations are not applicable. Actually, such existence results can be obtained through a combination of the result of this paper and the method of \emph{convex integration}. Since this task is rather long and complicated, we postpone it to the author's upcoming works. In this paper, we instead apply the main result to address large time behaviors of \emph{classical} solutions to such problems for suitable initial data.

We now state the main result of the paper as follows.

\begin{thm}\label{thm:uniform-conv}
Assume that $\sigma'>0$ in $\R$. Let $\rho\in C(\bar\Omega\times[0,\infty))\cap C^{2,1}(\Omega\times(0,\infty))$ be a solution to problem (\ref{main-ibP}),
where $\rho_0\in C(\bar\Omega)$ is an initial datum satisfying the compatibility condition $\rho_0(0)=\rho_0(L)=0$. Then for all $t\ge0,$ one has
\begin{equation}\label{thm:uniform-conv-rate}
\begin{split}
\|\rho(\cdot,t)\|_\infty\le & \|\rho_0\|_\infty \frac{\max\Big\{\frac{\|\rho_0\|_\infty \tilde\theta}{(1-\tau)\theta}+1,m \Big\}-e^{-\lambda L}}{\max\Big\{\frac{\|\rho_0\|_\infty \tilde\theta}{(1-\tau)\theta}+1,m \Big\}-1} \\
&\times \exp\bigg(-\frac{\tau\theta\lambda^2e^{-\lambda L}}{\max\Big\{\frac{\|\rho_0\|_\infty \tilde\theta}{(1-\tau)\theta}+1,m \Big\}-e^{-\lambda L}} t\bigg),
\end{split}
\end{equation}
where $\|\cdot\|_\infty:=\|\cdot\|_{L^\infty(\Omega)}$, $0<\tau<1$, $\lambda>0$, $m>1$,
\[
\theta:=\min_{[-\|\rho_0\|_\infty\frac{m}{m-1}, \|\rho_0\|_\infty\frac{m}{m-1}]}\sigma'\;\;\mbox{and}\;\;
\tilde\theta:=\max_{[-\|\rho_0\|_\infty\frac{m}{m-1}, \|\rho_0\|_\infty\frac{m}{m-1}]}|\sigma''|.
\]
\end{thm}

The constants $\tau$, $\lambda$ and $m$ in the theorem can be chosen arbitrarily as described to fix the rate of convergence; once these numbers are fixed, it is clear that the rate of convergence depends only on $\rho_0$, $\sigma$ and $L$.

The rest of the paper is organized as follows. In Section \ref{sec:max-com-principles}, we derive appropriate maximum and comparison principles for qualitative behaviors of solutions to problem (\ref{main-ibP}). Based on such a comparison principle, the proof of the main result, Theorem \ref{thm:uniform-conv}, is provided in Section \ref{sec:proof-main-thm}. As the last part of the paper, in Section \ref{sec:applications}, we apply the main result to the initial-boundary value problems for the heat equation,  two aggregation models in population dynamics and the Perona-Malik model in image processing.

In closing this section, we fix some notation. First, we will keep the notation, used in this section, throughout the paper unless otherwise stated. We use $W^{1,\infty}(\Omega)$ to denote the space of functions $v\in L^\infty(\Omega)$ with weak derivative $v'=v_x\in L^\infty(\Omega)$ and write the norm $\|v\|_{W^{1,\infty}(\Omega)}:=\|v\|_{L^\infty(\Omega)} +\|v_x\|_{L^\infty(\Omega)}$. Let $G\subset\R^2=\R_x\times\R_t$ be an open set. For integers $k,l\ge0$ with $2l\le k,$ we denote by $C^{k,l}(\bar{G})$ [resp. $C^{k,l}(G)$] the space of functions $v\in C(\bar{G})$ [$v\in C(G)$] with $\partial_t^i\partial_x^j v\in C(\bar{G})$ [$\partial_t^i\partial_x^j v\in C(G)$] for all integers $0\le i\le l$ and $0\le j+2i\le k$. For an integer $k\ge0$ and a number $\alpha\in(0,1)$, we write $C^{k+\alpha,\frac{k+\alpha}{2}}(\bar{G}):= H_{k+\alpha}(G)$, where $H_{k+\alpha}(G)$ is the parabolic H\"older space defined in \cite{Ln}.

\section{Maximum and comparison principles}\label{sec:max-com-principles}

In this section, we prepare two useful lemmas dealing with  maximum and comparison principles concerning problem (\ref{main-ibP}). Throughout this section, we assume that
\[
\sigma\in C^2(\R)\;\;\mbox{and}\;\;\sigma'\ge0\;\;\mbox{in $\R$.}
\]

\subsection{Maximum Principle}

The maximum principle below states that any classical solution $\rho$ to problem (\ref{main-ibP}) goes neither above the maximum value nor below the minimum value of the initial datum $\rho_0$ for all times.

\begin{lem}[Maximum Principle]\label{lem:maximum}
Let $\rho\in C(\bar\Omega\times[0,\infty))\cap C^{2,1}(\Omega\times(0,\infty))$ be a solution to problem (\ref{main-ibP}), where $\rho_0\in C(\bar\Omega)$ is an initial datum satisfying the compatibility condition
$
\rho_0(0)=\rho_0(L)=0.
$
Then
\[
\min_{\bar\Omega}\rho_0\le\rho\le\max_{\bar\Omega}\rho_0\;\;\mbox{in $\Omega\times(0,\infty)$}.
\]
\end{lem}

\begin{proof}
For simplicity, let us write $M_0=\max_{\bar\Omega}\rho_0$, $m_0=\min_{\bar\Omega}\rho_0$ and $\Omega_\tau=\Omega\times (0,\tau)$ for each $0<\tau\le\infty$; then $m_0\le 0\le M_0$. We also set $W(x,t)=\rho(x,t)-M_0$ and $w(x,t)=\rho(x,t)-m_0$ for all $(x,t)\in\bar\Omega_\infty.$ Then it is sufficient to show that
\[
W\le 0\;\;\mbox{and}\;\;w\ge 0\;\;\mbox{in $\Omega_\infty$}.
\]

We only prove that $w\ge 0$ in $\Omega_\infty$ as the other inequality $W\le0$ can be verified in the same way. To argue by contradiction, suppose there exists a point $(x_0,t_0)\in\Omega_\infty$ such that
\begin{equation}\label{lem:maximum-2}
w(x_0,t_0)<0.
\end{equation}
Fix any number $T\in(t_0,\infty),$ and define
\begin{equation}\label{lem:maximum-3}
z(x,t)=e^{-t}w(x,t)\;\;\mbox{for all $(x,t)\in\bar\Omega_T$};
\end{equation}
then $z(x,0)=\rho_0(x)-m_0\ge 0$ for all $x\in\bar\Omega$, and $z(0,t)=z(L,t)=-e^{-t}m_0\ge 0$ for all $t\in[0,T]$. This together with (\ref{lem:maximum-2}) implies that
\[
z(x_1,t_1)=\min_{\bar\Omega_T}z\le e^{-t_0}w(x_0,t_0)<0\;\;\mbox{for some $(x_1,t_1)\in\Omega\times(0,T]$}.
\]
From this, we have
\begin{equation}\label{lem:maximum-4}
\begin{split}
z_t(x_1,t_1) & \le 0,\;\; e^{-t_1}\rho_x(x_1,t_1)=z_x(x_1,t_1)=0,\;\;\mbox{and} \\
& e^{-t_1}\rho_{xx}(x_1,t_1)=z_{xx}(x_1,t_1)\ge0.
\end{split}
\end{equation}

From (\ref{lem:maximum-3}) and (\ref{lem:maximum-4}), we get
\[
0\ge z_t(x_1,t_1)=e^{-t_1}(w_t(x_1,t_1)-w(x_1,t_1)),
\]
yielding $w_t(x_1,t_1)\le w(x_1,t_1)=e^{t_1}z(x_1,t_1)<0$. On the other hand, from (\ref{main-ibP}), (\ref{lem:maximum-4}) and the assumption that $\sigma'\ge0$ in $\R$, we have
\[
\begin{split}
w_t(x_1,t_1)=\rho_t(x_1,t_1) & =\sigma'(\rho(x_1,t_1))\rho_{xx} (x_1,t_1)+\sigma''(\rho(x_1,t_1))(\rho_{x} (x_1,t_1))^2 \\
& = \sigma'(\rho(x_1,t_1))\rho_{xx} (x_1,t_1)\ge0.
\end{split}
\]
We thus arrive at a contradiction.
\end{proof}

\subsection{Comparison Principle}

The \emph{quasilinear} comparison principle below will be the pivotal tool for proving the main result of the paper, Theorem \ref{thm:uniform-conv}, in Section \ref{sec:proof-main-thm}.

\begin{lem}[Comparison Principle]\label{lem:comparison}
Assume that $v,w\in C(\bar\Omega\times[0,\infty))\cap C^{2,1}(\Omega\times(0,\infty))$ satisfy
\begin{equation}\label{lem:comparison-1}
\begin{split}
\mbox{either}\;\; & \left\{
\begin{array}{l}
  v_t\ge(\sigma(v))_{xx} \\
  w_t<(\sigma(w))_{xx}
\end{array}\right. \;\;\mbox{in $\Omega\times(0,\infty)$}\\
\mbox{or}\;\; &
\left\{
\begin{array}{l}
  v_t>(\sigma(v))_{xx} \\
  w_t\le(\sigma(w))_{xx}
\end{array}\right. \;\;\mbox{in $\Omega\times(0,\infty)$}
\end{split}
\end{equation}
and $v>w$ on $\partial(\Omega\times(0,\infty)).$
Then $v>w$ in $\Omega\times(0,\infty)$.
\end{lem}

\begin{proof}
To prove by contradiction, let us suppose that
\[
v(x_0,t_0)\le w(x_0,t_0)\;\;\mbox{for some $(x_0,t_0)\in\Omega\times(0,\infty)$}.
\]
Fix a number $T\in(t_0+1,\infty)$, and define
\[
z(x,t)=e^{-t}(v(x,t)-w(x,t))\;\;\forall(x,t)\in\bar\Omega\times[0,T].
\]
Then
\begin{equation}\label{lem:comparison-2}
z>0\;\;\mbox{on $(\Omega\times\{0\})\cup(\partial\Omega\times[0,T])$}\;\;\mbox{and} \;\; z(x_0,t_0)\le0.
\end{equation}

Define
\[
\mathcal{T}=\big\{t\in[0,T]\,|\,z>0\;\mbox{on}\;\bar\Omega\times[0,t]\big\}\;\; \mbox{and}\;\;t^*=\sup\mathcal{T}.
\]
By (\ref{lem:comparison-2}), we have $0\in\mathcal{T}$ and $t\not\in\mathcal{T}$ for all $t_0\le t\le T$; hence $\{0\}\subset\mathcal{T}\subset[0,t_0)$, and so $0\le t^*\le t_0<T-1$. By the definition of $t^*$, for each $j\in\N$, we have
\[
z(x_j,t_j)\le 0\;\;\mbox{for some $(x_j,t_j)\in\bar{\Omega}\times[t^*,t^*+1/j)$}.
\]
Passing to a subsequence if necessary, we have $x_j\to x^*$ as $j\to\infty$ for some $x^*\in\bar\Omega$; then $z(x^*,t^*)\le0$. From this and (\ref{lem:comparison-2}), we now see that $x^*\in\Omega$ and $0<t^*\le t_0.$ Also, by the definition of $t^*$, we have $z>0$ on $\bar\Omega\times[0,t^*)$. So we easily deduce that at the point $(x,t)=(x^*,t^*)$,
\begin{equation}\label{lem:comparison-3}
z=0,\;\;z_x=0,\;\;z_{xx}\ge0\;\;\mbox{and}\;\;z_t\le0;
\end{equation}
thus  at $(x,t)=(x^*,t^*)$,
\begin{equation}\label{lem:comparison-4}
v=w,\;\;v_x=w_x\;\;\mbox{and}\;\;v_{xx}\ge w_{xx}.
\end{equation}

On one hand, we have from (\ref{lem:comparison-3}) that
\[
0\ge z_t(x^*,t^*)=e^{-t^*}((v-w)_t(x^*,t^*)-(v-w)(x^*,t^*));
\]
thus by (\ref{lem:comparison-4}), we get
\[
(v-w)_t(x^*,t^*)\le(v-w)(x^*,t^*)=0.
\]
On the other hand, it follows from (\ref{lem:comparison-1}), (\ref{lem:comparison-4}) and $\sigma'\ge0$ that
\[
\begin{split}
(v-w)_t(x^*,t^*)> & \sigma'(v(x^*,t^*))v_{xx}(x^*,t^*)- \sigma'(w(x^*,t^*))w_{xx}(x^*,t^*) \\
& +\sigma''(v(x^*,t^*))(v_x(x^*,t^*))^2- \sigma''(w(x^*,t^*))(w_x(x^*,t^*))^2 \\
= & \sigma'(v(x^*,t^*))(v_{xx}(x^*,t^*)-w_{xx}(x^*,t^*))\ge 0;
\end{split}
\]
that is, $(v-w)_t(x^*,t^*)>0$.
We thus have a contradiction.
\end{proof}

%\begin{eg}
%As an illustration, let us apply the previous result to the simplest case of heat equation with $\Omega=(0,1)$; that is, $\sigma(\rho)\equiv\rho.$ In this case, we have $\theta=1$, $\tilde\theta_0=0$ and $L=1$. It thus follows that for all $t\ge0$,
%\[
%\|\rho(\cdot,t)\|_\infty\le \|\rho_0\|_\infty\frac{m-e^{-\lambda}}{m-1}e^{-\frac{\tau\lambda^2 e^{-\lambda}}{m-e^{-\lambda}}t},
%\]
%where $0<\tau<1$, $\lambda>0$ and $m>1$ are arbitrary. Letting $\tau\to 1^-,$ we have that for all $t\ge 0$,
%\[
%\|\rho(\cdot,t)\|_\infty\le \|\rho_0\|_\infty\frac{m-e^{-\lambda}}{m-1}e^{-\frac{\lambda^2 e^{-\lambda}}{m-e^{-\lambda}}t}.
%\]

%\end{eg}

\section{Proof of Theorem \ref{thm:uniform-conv}}\label{sec:proof-main-thm}
%In this section, we complete the proof of Theorem \ref{thm:uniform-conv} and Corollary \ref{coro:uniform-conv}.

%\subsection{Proof of Theorem \ref{thm:uniform-conv}}
Using the comparison principle, Lemma \ref{lem:comparison}, we now prove our main result, Theorem \ref{thm:uniform-conv}.

To start the proof, fix any $0<\tau<1$, $\lambda>0$ and $m>1$. Let $0<\epsilon\le 1$. Define
\[
\varphi(x)=\varphi_{s,\lambda}(x)=s-e^{-\lambda  x}\;\;\mbox{for all $x\in\bar\Omega$,}
\]
where $s\in[m,\infty)$ is a constant to be specified later. Then
\[
\delta_0:=\min_{\bar\Omega}\varphi=s-1>0\;\;\mbox{and} \;\;\delta_1:=\max_{\bar\Omega}\varphi=s-e^{-\lambda L}.
\]
Next, we define that for all $(x,t)\in\bar\Omega\times[0,\infty),$
\[
\psi(x,t)  =\psi_{s,\lambda,\gamma,A}(x,t)=A\frac{\varphi(x)}{\delta_0}e^{-\gamma t}\;\;\mbox{and}\;\;
\tilde\psi(x,t)  =-\psi(x,t).
\]
where $\gamma>0$ and $A>0$ are constants to be chosen below.
Observe that for all $x\in\bar\Omega,$
\[
\psi(x,0)= A\frac{\varphi(x)}{\delta_0} \ge A\;\;\mbox{and}\;\;\tilde\psi(x,0)=-\psi(x,0)\le -A
\]
and that for all $t>0$,
\[
\psi(0,t)=Ae^{-\gamma t}>0,\;\; \psi(L,t)=A\frac{\delta_1}{\delta_0}e^{-\gamma t} >0,
\]
\[
\tilde\psi(0,t)=-\psi(0,t)<0\;\;\mbox{and}\;\;\tilde\psi(L,t) =-\psi(L,t)<0.
\]
We choose $A=A_\epsilon=\|\rho_0\|_\infty+\epsilon$; then it  follows that
\begin{equation}\label{thm:uniform-conv-1}
\tilde\psi<\rho<\psi\;\;\mbox{on $\partial(\Omega\times(0,\infty))$}.
\end{equation}

Let $(x,t)\in\Omega\times(0,\infty)$. We first compute
\[
\psi_x(x,t)  = \frac{A_\epsilon}{\delta_0}\lambda e^{-\lambda x}e^{-\gamma t},\;\; \psi_{xx}(x,t)  =  -\frac{A_\epsilon}{\delta_0}\lambda^2 e^{-\lambda x}e^{-\gamma t}\;\;\mbox{and}
\]
\[
\psi_t(x,t)  =-\frac{A_\epsilon}{\delta_0}\gamma(s-e^{-\lambda x})e^{-\gamma t}.
\]
Using these, we get
\[
\begin{split}
\mathcal{L}\psi(x,t)  := & -\psi_t(x,t) +(\sigma(\psi(x,t)))_{xx} \\
= & -\psi_t(x,t) +\sigma'(\psi(x,t))\psi_{xx}(x,t)+\sigma''(\psi(x,t))(\psi_x(x,t))^2\\
= & \frac{A_\epsilon}{\delta_0}\gamma(s-e^{-\lambda x})e^{-\gamma t} -\sigma'(\psi(x,t)) \frac{A_\epsilon}{\delta_0}\lambda^2 e^{-\lambda x}e^{-\gamma t} \\
& + \sigma''(\psi(x,t)) \frac{A_\epsilon^2}{\delta^2_0}\lambda^2 e^{-2\lambda x}e^{-2\gamma t}
\end{split}
\]
and
\[
\begin{split}
\mathcal{L}\tilde\psi(x,t) = &  -\frac{A_\epsilon}{\delta_0}\gamma(s-e^{-\lambda x})e^{-\gamma t} +\sigma'(\tilde\psi(x,t)) \frac{A_\epsilon}{\delta_0}\lambda^2 e^{-\lambda x}e^{-\gamma t}  \\
& + \sigma''(\tilde\psi(x,t)) \frac{A_\epsilon^2}{\delta_0^2}\lambda^2 e^{-2\lambda x}e^{-2\gamma t}.
\end{split}
\]
From the definition of $\psi$ and the choice $s\ge m>1$, we see that
\[
0<\psi(x,t)\le A_\epsilon\frac{\delta_1}{\delta_0}= A_\epsilon\frac{s-e^{-\lambda L}}{s-1}< A_\epsilon\frac{s}{s-1} \le A_\epsilon\frac{m}{m-1}=:R_{\epsilon}
\]
and that
\[
-R_\epsilon<\tilde\psi(x,t)=-\psi(x,t)<0.
\]
We set
\[
\theta_{\epsilon}=\min_{[-R_\epsilon,R_\epsilon]}\sigma'>0\;\;\mbox{and}\;\; \tilde\theta_{\epsilon}=\max_{[-R_\epsilon,R_\epsilon]}|\sigma''|;
\]
then
\[
\begin{split}
\mathcal{L}\psi(x,t) \le & \frac{A_\epsilon}{\delta_0}\gamma(s-e^{-\lambda x})e^{-\gamma t} - \theta_\epsilon\frac{A_\epsilon}{\delta_0}\lambda^2 e^{-\lambda x}e^{-\gamma t} +\tilde\theta_\epsilon \frac{A_\epsilon^2}{\delta_0^2}\lambda^2 e^{-2\lambda x}e^{-2\gamma t} \\
= & \Big(\frac{A_\epsilon}{\delta_0}\gamma (s-e^{-\lambda x}) e^{-\gamma t} - \tau\frac{\theta_\epsilon A_\epsilon}{\delta_0}\lambda^2 e^{-\lambda x}e^{-\gamma t}\Big)  \\
& + \Big(  \frac{\tilde\theta_\epsilon A_\epsilon^2}{\delta_0^2}\lambda^2 e^{-2\lambda x}e^{-2\gamma t} - (1-\tau)\frac{\theta_\epsilon A_\epsilon}{\delta_0}\lambda^2 e^{-\lambda x}e^{-\gamma t} \Big) \\
=&:  I_1+I_2
\end{split}
\]
and
\[
\begin{split}
\mathcal{L}\tilde\psi(x,t) \ge & -\frac{A_\epsilon}{\delta_0}\gamma(s-e^{-\lambda x})e^{-\gamma t} + \theta_\epsilon \frac{A_\epsilon}{\delta_0}\lambda^2 e^{-\lambda x}e^{-\gamma t} -\tilde\theta_\epsilon \frac{A_\epsilon^2}{\delta_0^2}\lambda^2 e^{-2\lambda x}e^{-2\gamma t}  \\
= & \Big(-\frac{A_\epsilon}{\delta_0}\gamma (s-e^{-\lambda x}) e^{-\gamma t} +  \tau\frac{\theta_\epsilon A_\epsilon}{\delta_0}\lambda^2 e^{-\lambda x}e^{-\gamma t}\Big)  \\
& + \Big(  -\frac{\tilde\theta_\epsilon A_\epsilon^2}{\delta_0^2}\lambda^2 e^{-2\lambda x}e^{-2\gamma t} + (1-\tau)\frac{\theta_\epsilon A_\epsilon}{\delta_0}\lambda^2 e^{-\lambda x}e^{-\gamma t} \Big) \\
=&:  J_1+J_2.
\end{split}
\]

We now control the quantities $I_1=-J_1$ and $I_2=-J_2$.  Since $\lambda x>0$ and $\gamma t>0$, we have
\[
\begin{split}
I_2 & \le \frac{\tilde\theta_{\epsilon} A_\epsilon^2}{\delta_0^2}\lambda^2 e^{-\lambda x}e^{-\gamma t} - (1-\tau)\frac{\theta_\epsilon A_\epsilon}{\delta_0}\lambda^2 e^{-\lambda x}e^{-\gamma t} \\
& = \frac{A_\epsilon\lambda^2 e^{-\lambda x}e^{-\gamma t}}{\delta_0}\Big( \frac{A_\epsilon\tilde\theta_{\epsilon}}{\delta_0}-(1-\tau)\theta_\epsilon \Big)<0
\end{split}
\]
and $J_2=-I_2>0$
provided that the constant $s\in[m,\infty)$ is chosen so large that
\begin{equation}\label{thm:uniform-conv-2}
s>\frac{A_\epsilon\tilde\theta_{\epsilon}}{(1-\tau)\theta_\epsilon}+1.
\end{equation}
Next, we have
\[
I_1 =  \frac{A_\epsilon e^{-\gamma t}}{\delta_0}(\gamma (s-e^{-\lambda x}) - \tau\theta_\epsilon\lambda^2 e^{-\lambda x})\le 0\;\;\mbox{and} \;\; J_1=-I_1\ge 0
\]
if the constant $\gamma>0$ is chosen so small that
\begin{equation}\label{thm:uniform-conv-3}
\gamma\le\frac{\tau\theta_\epsilon\lambda^2e^{-\lambda L}}{s-e^{-\lambda L}}.
\end{equation}

In summary, if we let
\[
%\begin{split}
%0 <\epsilon\le1  ,\;\;0<\tau<1,\;\;\lambda>0, \\
A=A_\epsilon  =\|\rho_0\|_\infty+\epsilon,\;\;R_{\epsilon}  =A_\epsilon\frac{m}{m-1},\;\;\theta_{\epsilon} =\min_{[-R_{\epsilon},R_{\epsilon}]}\sigma', \;\;\tilde\theta_{\epsilon} =\max_{[-R_{\epsilon},R_{\epsilon}]}|\sigma''|,
\]
\[
s=s_\epsilon=\max\Big\{\frac{A_\epsilon\tilde\theta_{\epsilon}} {(1-\tau)\theta_\epsilon}+1,m \Big\}+\epsilon, \;\;\mbox{and} \;\;
\gamma=\gamma_\epsilon  =\frac{\tau\theta_\epsilon\lambda^2e^{-\lambda L}}{s_\epsilon-e^{-\lambda L}},
%\end{split}
\]
then (\ref{thm:uniform-conv-2}) and (\ref{thm:uniform-conv-3}) are satisfied so that  for all $(x,t)\in\Omega\times(0,\infty),$
\[
\mathcal{L}\psi(x,t)\le I_1+I_2 <0< J_1+J_2 \le \mathcal{L}\tilde\psi(x,t).
\]
With this and $(\ref{thm:uniform-conv-1})$, we can apply Lemma \ref{lem:comparison} to obtain that for all $(x,t)\in\Omega\times(0,\infty)$,
\[
-\psi(x,t)=\tilde\psi(x,t)<\rho(x,t)<\psi(x,t),
\]
that is,
\[
|\rho(x,t)|<\psi(x,t)\le A_\epsilon\frac{\delta_1}{\delta_0} e^{-\gamma_\epsilon t} =A_\epsilon\frac{s_\epsilon-e^{-\lambda L}}{s_\epsilon-1} e^{-\gamma_\epsilon t}.
\]
Letting $\epsilon\to 0^+,$ it follows that for all $(x,t)\in\Omega\times(0,\infty)$,
\[
\begin{split}
|\rho(x,t)|\le & \|\rho_0\|_\infty \frac{\max\Big\{\frac{\|\rho_0\|_\infty \tilde\theta}{(1-\tau)\theta}+1,m \Big\}-e^{-\lambda L}}{\max\Big\{\frac{\|\rho_0\|_\infty \tilde\theta}{(1-\tau)\theta}+1,m \Big\}-1} \\
&\times \exp\bigg(-\frac{\tau\theta\lambda^2e^{-\lambda L}}{\max\Big\{\frac{\|\rho_0\|_\infty \tilde\theta}{(1-\tau)\theta}+1,m \Big\}-e^{-\lambda L}} t\bigg),
\end{split}
\]
where
\[
\theta:=\min_{[-\|\rho_0\|_\infty\frac{m}{m-1}, \|\rho_0\|_\infty\frac{m}{m-1}]}\sigma'\;\;\mbox{and}\;\;
\tilde\theta:=\max_{[-\|\rho_0\|_\infty\frac{m}{m-1}, \|\rho_0\|_\infty\frac{m}{m-1}]}|\sigma''|.
\]

The proof of Theorem \ref{thm:uniform-conv} is now complete.

\section{Applications}\label{sec:applications}

In this section, we apply the main result, Theorem \ref{thm:uniform-conv}, to the initial-boundary value problems concerning  three types of equations  in one space dimension. In particular, we are mainly interested in estimating an exponential rate of convergence for evolutions arising in population dynamics and image processing.

Throughout this section, let $\alpha\in(0,1)$ be any fixed number representing a H\"older exponent.

\subsection{Heat diffusion}

As a model example, we consider the initial value problem of the heat equation
\begin{equation}\label{Heat-iP}
\left\{
\begin{array}{ll}
  \rho_t=\rho_{xx} & \mbox{in $\Omega\times (0,\infty)$,} \\
  \rho=\rho_0 & \mbox{on $\Omega\times\{t=0\}$},
\end{array}
 \right.
\end{equation}
coupled with either the Dirichlet boundary condition
\begin{equation}\label{Heat-D-bdry}
\rho=0\quad\mbox{on $\partial\Omega\times(0,\infty)$}
\end{equation}
or the Neumann boundary condition
\begin{equation}\label{Heat-N-bdry}
\rho_x=0\quad\mbox{on $\partial\Omega\times(0,\infty)$},
\end{equation}
where $\Omega=(0,1)\subset\R$, $\rho_0=\rho_0(x)$ is a given initial datum, and $\rho=\rho(x,t)$ is a solution to the problem.

Let us first assume that $\rho\in C(\bar\Omega\times[0,\infty))\cap C^{2,1}(\Omega\times(0,\infty))$ is a solution to problem (\ref{Heat-iP})(\ref{Heat-D-bdry}), where $\rho_0\in C(\bar\Omega)$ satisfies the compatibility condition $\rho_0(0)=\rho_0(1)=0.$ Following the notation of Theorem \ref{thm:uniform-conv}, we now have $L=1$, $\theta=1$ and $\tilde\theta=0$; thus letting $\tau\to 1^-$ to the result of the theorem, we obtain that for all $t\ge 0$,
\[
\|\rho(\cdot,t)\|_\infty\le \|\rho_0\|_\infty \frac{m-e^{-\lambda}}{m-1} e^{-\frac{\lambda^2 e^{-\lambda}}{m-e^{-\lambda}}t},
\]
where $\lambda>0$ and $m>1$ are arbitrary. Here, the least upper bound for the rates $\gamma(\lambda,m)=\frac{\lambda^2 e^{-\lambda}}{m-e^{-\lambda}}$ with $\lambda>0$ and $m>1$ lies in the interval $(0.64,0.65)$. However, in this case, it is well known  from the explicit formula \cite[Chapter 6]{Fr} that the exponential rate of convergence for the solution $\rho$ can be $\pi^2$, which is much larger than our value. Thus our result may not give an optimal rate of convergence.

Next, we assume that $\rho\in C^{2,1}(\bar\Omega\times[0,\infty))$ is a solution to problem (\ref{Heat-iP})(\ref{Heat-N-bdry}), where $\rho_0\in C^2(\bar\Omega)$ satisfies the compatibility condition $\rho_0'(0)=\rho_0'(1)=0.$ By the smoothing effect of the heat diffusion, we know that $\rho$ is smooth in $\Omega\times(0,\infty).$ Let $w=\rho_x\in C^{1,0}(\bar\Omega\times[0,\infty))\cap C^\infty(\Omega\times(0,\infty))$; then, with $w_0=\rho_0'\in C^1(\bar\Omega)$, it is easy to see that $w$ solves problem (\ref{Heat-iP})(\ref{Heat-D-bdry}), where $\rho$ and $\rho_0$ are replaced by $w$ and $w_0$, respectively. Therefore, we can use the above result and Poincar\'e's inequality to conclude that
\[
\|\rho(\cdot,t)-\bar\rho_0\|_{W^{1,\infty}(\Omega)}\le C e^{-\gamma t}\quad\forall t\ge0,
\]
where $\bar\rho_0:=\int_0^1\rho_0(x)\,dx$, $\lambda>0$, $m>1$, $\gamma=\frac{\lambda^2 e^{-\lambda}}{m-e^{-\lambda}}$, and $C>0$ is a constant depending only on $\|\rho_0'\|_\infty$, $m$ and $\lambda$.

\subsection{Aggregative movement in population dynamics}

The evolution process for spatial distribution of animals or biological organisms in one-dimensional homogeneous habitat can be modeled by quasilinear advection-diffusion equations of the form
\begin{equation}\label{Population-P}
\rho_t=(\sigma(\rho))_{xx},
\end{equation}
where $\rho=\rho(x,t)$ denotes the population density of a single species at position $x$ and time $t$, and the derivative of a given flux function $\sigma=\sigma(s)$ is the \emph{diffusivity} of equation (\ref{Population-P}).

As an alternative to the $\Delta$-model proposed by {Taylor and Taylor} \cite{TT} for modeling aggregative movement, {Turchin} \cite{Tu} derived equation (\ref{Population-P}), based on a random walk approach \cite{Ok}, as a model of individual movement, which is not only reflecting aggregation or repulsion between conspecific organisms but also avoiding some defects in their model. In his model, the flux function $\sigma$ is given by
\begin{equation*}%\label{Population-P-Turchin}
\sigma(s)=\frac{2k_0}{3\omega}s^3 -k_0 s^2 +\frac{\mu}{2} s,
\end{equation*}
where $k_0>0$ is the maximum degree of gregariousness, $\omega>0$ is the critical density at which movement switches from aggregative to repulsive, and $\mu\in(0,1]$ is the motility rate.

On the other hand, {Anguige and Schmeiser} \cite{AS} independently obtained equation (\ref{Population-P}), based also on the random walk approach, as a model of cell motility which incorporates the effects of cell-to-cell adhesion and volume filling. In their model, the flux function $\sigma$ is given by
\begin{equation}\label{Population-P-AS}
\sigma(s)=a s^3-2a s^2+s,
\end{equation}
where $a\in [0,1]$ is the adhesion constant.

For definiteness, let us adopt the flux function $\sigma$ in (\ref{Population-P-AS}) and consider the initial-boundary value problem
\begin{equation}\label{Population-ibP}
\left\{
\begin{array}{ll}
  \rho_t=(\sigma(\rho))_{xx} & \mbox{in $\Omega\times(0,\infty)$}, \\
  \rho=\rho_0 & \mbox{on $\Omega\times\{t=0\}$}, \\
  \rho(0,t)=\rho(L,t)=0 & \mbox{for $t\in(0,\infty)$},
\end{array}\right.
\end{equation}
where $\Omega=(0,L)\subset\R$ is a favorable habitat of size $L>0$, and $\rho_0(x)$ is the initial density of a given species at position $x$. Here, the \emph{absorbing} boundary condition $\rho(0,t)=\rho(L,t)=0$ means in a viewpoint of ecology that animals touching the border $\partial\Omega=\{0,L\}$ are permanently lost to the population, either because they move away from the habitat $\Omega$ or because they are killed by predators residing in the very hostile surrounding area $\R\setminus\bar\Omega$.

Note that the diffusivity $\sigma'$ is
\[
\sigma'(s)=3a s^2-4a s+1=3a\Big(s-\frac{2}{3}\Big)^2 +1-\frac{4}{3}a.
\]
In case of a weak adhesion effect with $0\le a<\frac{3}{4}$, the diffusivity $\sigma'$ is positive everywhere with absolute minimum value $1-\frac{4}{3}a>0$, and so problem (\ref{Population-ibP}) is well-posed and admits a global classical solution $\rho$ for all sufficiently smooth initial data $\rho_0$ with $\rho_0(0)=\rho_0(L)=0$. In a highly aggregative species with $\frac{3}{4}<a\le1,$ since the diffusivity $\sigma'$ can take both positive and negative values, (\ref{Population-ibP}) is ill-posed and may not even possess a local classical solution for some smooth initial data. The critical adhesion constant $a=\frac{3}{4}$ makes (\ref{Population-ibP}) \emph{degenerate parabolic}; we do not handle this case here.

To be more specific, let us first consider the case that $0\le a<\frac{3}{4}.$ Let $\rho_0\in C^{2+\alpha}(\bar\Omega)$ be such that $\rho_0\ge0$ in $\Omega$ and $\rho_0(0)=\rho_0(L)=0$. Then it follows from \cite[Theorem 12.14]{Ln} that there exists a unique solution $\rho\in C^{2,1}(\bar\Omega\times[0,\infty))$ to problem (\ref{Population-ibP}) such that $\rho\in C^{2+\alpha,1+\frac{\alpha}{2}}(\bar\Omega\times[0,T])$ for each $T>0.$ From Lemma \ref{lem:maximum} and the initial and boundary conditions, we have that for all $t_2>t_1\ge0$,
\[
0=\min_{\bar\Omega}\rho(\cdot,t_1)=\min_{\bar\Omega}\rho(\cdot,t_2)\le \max_{\bar\Omega}\rho(\cdot,t_2)\le \max_{\bar\Omega}\rho(\cdot,t_1).
\]
To apply Theorem \ref{thm:uniform-conv}, fix any $0<\tau<1$, $\lambda>0$ and $m>1$. Following the notation of the theorem, we easily get
\[
\theta=\left\{
\begin{array}{ll}
  3a(\|\rho_0\|_\infty\frac{m}{m-1}-\frac{2}{3})^2 +1-\frac{4}{3}a & \mbox{if $\|\rho_0\|_\infty\frac{m}{m-1}<\frac{2}{3}$,} \\
  1-\frac{4}{3}a & \mbox{if $\|\rho_0\|_\infty\frac{m}{m-1}\ge\frac{2}{3}$,}
\end{array}
 \right.
\]
and
\[
\tilde\theta=2a\Big(3\|\rho_0\|_\infty\frac{m}{m-1}+2\Big).
\]
With these numbers, rate of convergence (\ref{thm:uniform-conv-rate}) follows from Theorem \ref{thm:uniform-conv}.

Next, assume that $\frac{3}{4}<a\le 1.$ To deal with only \emph{classical} solutions, let $\rho_0\in C^{2+\alpha}(\bar\Omega)$ be such that $\rho_0\ge0$ in $\Omega$, $\rho_0(0)=\rho_0(L)=0$, and
\[
0<b_0:=\|\rho_0\|_\infty< \frac{2a-\sqrt{a(4a-3)}}{3a};
\]
then $\frac{2a-\sqrt{a(4a-3)}}{3a b_0}>1$ so that there exists a unique number $m^*>1$ with $\frac{m^*}{m^*-1}= \frac{2a-\sqrt{a(4a-3)}}{3a b_0}$, that is, $m^*=\frac{2a-\sqrt{a(4a-3)}} {a(2-3b_0)-\sqrt{a(4a-3)}}$. Let fix any $0<\tau<1$, $\lambda>0$ and $m>m^*;$ here, $b_0\frac{m}{m-1}< \frac{2a-\sqrt{a(4a-3)}}{3a}$. Set $\bar{s}=\frac{1}{2}\big(b_0\frac{m}{m-1}+ \frac{2a-\sqrt{a(4a-3)}}{3a}\big)$. We then modify the function $\sigma(s)=a s^3-2a s^2+s$ so as to obtain a function $\tilde\sigma\in C^3(\R)$ such that
\begin{equation}\label{Population-1}
\left\{
\begin{array}{l}
  \tilde\sigma(s)=\sigma(s)\;\;\forall s\in[-\bar{s},\bar{s}], \\
  \lambda\le\tilde\sigma'(s)\le\Lambda\;\;\forall s\in\R,  \\
  \mbox{$\tilde\sigma'''$ is bounded in $\R$,}
\end{array}
\right.
\end{equation}
where $\Lambda>\lambda>0$ are some constants. It now follows from \cite[Theorem 12.14]{Ln} that there exists a unique solution $\rho\in C^{2,1}(\bar\Omega\times[0,\infty))$ to the initial-boundary value problem
\begin{equation}\label{Population-2}
\left\{
\begin{array}{ll}
  \rho_t=(\tilde\sigma(\rho))_{xx} & \mbox{in $\Omega\times(0,\infty)$}, \\
  \rho=\rho_0 & \mbox{on $\Omega\times\{t=0\}$}, \\
  \rho(0,t)=\rho(L,t)=0 & \mbox{for $t\in(0,\infty)$},
\end{array}
\right.
\end{equation}
such that $\rho\in C^{2+\alpha,1+\frac{\alpha}{2}}(\bar\Omega\times[0,T])$ for each $T>0$. From Lemma \ref{lem:maximum} and the initial and boundary conditions, we have that for all $t_2>t_1\ge0$,
\[
0=\min_{\bar\Omega}\rho(\cdot,t_1)=\min_{\bar\Omega}\rho(\cdot,t_2)\le \max_{\bar\Omega}\rho(\cdot,t_2)\le \max_{\bar\Omega}\rho(\cdot,t_1).
\]
In particular, we have $0\le \rho\le \|\rho_0\|_{\infty}$ in $\Omega\times(0,\infty)$; thus from (\ref{Population-1}) and (\ref{Population-2}), we see that $\rho$ is also a solution to problem (\ref{Population-ibP}). Using the notation of Theorem \ref{thm:uniform-conv}, it follows from (\ref{Population-1}) that
\[
\theta=  3a\Big(\|\rho_0\|_\infty\frac{m}{m-1}-\frac{2}{3}\Big)^2 +1-\frac{4}{3}a
\]
and
\[
\tilde\theta=2a\Big(3\|\rho_0\|_\infty\frac{m}{m-1}+2\Big).
\]
With these numbers, rate of convergence (\ref{thm:uniform-conv-rate}) follows from the theorem.

We now summarize what we have discussed so far as follows. The first one is on the large time behaviors of classical solutions to problem (\ref{Population-ibP}) under a weak adhesion effect for all smooth and positive initial data.

\begin{thm}[Weak aggregation]\label{thm:application-Population-wk}
Let $\sigma$ be the flux function given by (\ref{Population-P-AS}). Assume $0\le a<\frac{3}{4}$. Let $\rho_0\in C^{2+\alpha}(\bar\Omega)$ be such that $\rho_0\ge 0$ in $\Omega$ and $\rho_0(0)=\rho_0(L)=0$. Then there exists a unique solution $\rho\in C^{2,1}(\bar\Omega\times[0,\infty))$ to problem (\ref{Population-ibP}) satisfying the following:
\begin{itemize}
\item[(i)] $\rho\in C^{2+\alpha,1+\frac{\alpha}{2}}(\bar\Omega\times[0,T])$ for each $T>0$,
\item[(ii)] $0=\min_{\bar\Omega}\rho(\cdot,t_1)=\min_{\bar\Omega}\rho(\cdot,t_2)\le \max_{\bar\Omega}\rho(\cdot,t_2)\le \max_{\bar\Omega}\rho(\cdot,t_1)$ for all $t_2>t_1\ge0$,
\item[(iii)] $\|\rho(\cdot,t)\|_\infty\le C e^{-\gamma t}$ for all $t\ge 0$, where $0<\tau<1$, $\lambda>0$, $m>1$,
\[
\theta:=\left\{
\begin{array}{ll}
  3a(\|\rho_0\|_\infty\frac{m}{m-1}-\frac{2}{3})^2 +1-\frac{4}{3}a & \mbox{if $\|\rho_0\|_\infty\frac{m}{m-1}<\frac{2}{3}$,} \\
  1-\frac{4}{3}a & \mbox{if $\|\rho_0\|_\infty\frac{m}{m-1}\ge\frac{2}{3}$,}
\end{array}
 \right.
\]
\[
\tilde\theta:=2a\Big(3\|\rho_0\|_\infty\frac{m}{m-1}+2\Big),
\]
\[
\gamma:= \frac{\tau\theta\lambda^2e^{-\lambda L}}{\max\Big\{\frac{\|\rho_0\|_\infty \tilde\theta}{(1-\tau)\theta}+1,m \Big\}-e^{-\lambda L}},\;\;\mbox{and}
\]
\[
C:= \|\rho_0\|_\infty \frac{\max\Big\{\frac{\|\rho_0\|_\infty \tilde\theta}{(1-\tau)\theta}+1,m \Big\}-e^{-\lambda L}}{\max\Big\{\frac{\|\rho_0\|_\infty \tilde\theta}{(1-\tau)\theta}+1,m \Big\}-1}.
\]
\end{itemize}
\end{thm}

The next one deals with the large time behaviors of classical solutions to problem (\ref{Population-ibP}) in a strongly aggregative species  for all smooth and positive initial data whose maximum is less than a certain threshold.

\begin{thm}[Strong aggregation]\label{thm:application-Population-strong}
Let $\sigma$ be the flux function given by (\ref{Population-P-AS}). Assume $\frac{3}{4}<a\le 1$. Let $\rho_0\in C^{2+\alpha}(\bar\Omega)$ be such that $\rho_0\ge 0$ in $\Omega$, $\rho_0(0)=\rho_0(L)=0$, and
\[
\|\rho_0\|_\infty< \frac{2a-\sqrt{a(4a-3)}}{3a}.
\]
Then there exists a unique solution $\rho\in C^{2,1}(\bar\Omega\times[0,\infty))$ to problem (\ref{Population-ibP})  satisfying the following:
\begin{itemize}
\item[(i)] $\rho\in C^{2+\alpha,1+\frac{\alpha}{2}}(\bar\Omega\times[0,T])$ for each $T>0$,
\item[(ii)] $0=\min_{\bar\Omega}\rho(\cdot,t_1)=\min_{\bar\Omega}\rho(\cdot,t_2)\le \max_{\bar\Omega}\rho(\cdot,t_2)\le \max_{\bar\Omega}\rho(\cdot,t_1)$ for all $t_2>t_1\ge0$,
\item[(iii)] $\|\rho(\cdot,t)\|_\infty\le C e^{-\gamma t}$ for all $t\ge 0$, where $0<\tau<1$, $\lambda>0$,
\[
m> \frac{2a-\sqrt{a(4a-3)}} {a(2-3\|\rho_0\|_\infty)-\sqrt{a(4a-3)}},
\]
\[
\theta:=3a\Big(\|\rho_0\|_\infty\frac{m}{m-1}-\frac{2}{3}\Big)^2 +1-\frac{4}{3}a,
\]
\[
\tilde\theta:=2a\Big(3\|\rho_0\|_\infty\frac{m}{m-1}+2\Big),
\]
\[
\gamma:= \frac{\tau\theta\lambda^2e^{-\lambda L}}{\max\Big\{\frac{\|\rho_0\|_\infty \tilde\theta}{(1-\tau)\theta}+1,m \Big\}-e^{-\lambda L}},\;\;\mbox{and}
\]
\[
C:= \|\rho_0\|_\infty \frac{\max\Big\{\frac{\|\rho_0\|_\infty \tilde\theta}{(1-\tau)\theta}+1,m \Big\}-e^{-\lambda L}}{\max\Big\{\frac{\|\rho_0\|_\infty \tilde\theta}{(1-\tau)\theta}+1,m \Big\}-1}.
\]
\end{itemize}
\end{thm}

\begin{remk}
We now interpret the result of Theorems \ref{thm:application-Population-wk} and \ref{thm:application-Population-strong} in the context of zoology. Imposing the absorbing boundary condition, the outside environment is very hostile so that animals crossing the border of the habitat will be removed instantly. This condition would readily imply the extinction of the species after a long time. Such an expected phenomenon is easily confirmed when gregariousness among the animals is low (Theorem \ref{thm:application-Population-wk}): the population density converges uniformly to $0$ at an exponential rate as $t\to\infty$. However, when the aggregation effect between the animals is high, it is not a simple task to establish the existence of global \emph{weak} solutions to problem (\ref{Population-ibP}) for initial data whose magnitude exceeds a certain threshold. Nonetheless, the expected phenomenon is verified if the maximum of the initial datum is less than the threshold (Theorem \ref{thm:application-Population-strong}).
\end{remk}

\subsection{The Perona-Malik model}
As the last application, we consider the Perona-Malik model \cite{PM} in image processing:
\begin{equation}\label{PM-ibP}
\left\{
\begin{array}{ll}
  u_t=\big(\frac{u_x}{1+u_x^2}\big)_{x} & \mbox{in $\Omega\times(0,\infty)$,} \\
  u=u_0 & \mbox{on $\Omega\times\{t=0\}$,} \\
  u_x(0,t)=u_x(1,t)=0 & \mbox{for $t\in(0,\infty)$},
\end{array}\right.
\end{equation}
where $\Omega=(0,1)\subset\R$ is the spatial domain for a one-dimensional computer vision, $u_0(x)$ is the grey level of an initial \emph{noisy} picture at point $x\in\Omega$, and $u(x,t)$ denotes the grey level of the \emph{enhanced} image at point $x\in\Omega$ and time $t>0$.

Problem (\ref{PM-ibP}) is well-posed and admits a global classical solution for all sufficiently smooth initial data $u_0$ with $u_0'(0)=u_0'(1)=0$ and $\|u_0'\|_\infty<1$  \cite{KK}. If $u_0'(0)=u_0'(1)=0$ and $\|u_0'\|_\infty>1$, no global $C^1$ solution to (\ref{PM-ibP}) exists \cite{KK, Go}; in this case,  even a local $C^1$ solution would not exist unless $u_0$ were infinitely differentiable \cite{Ky}.

Here, we focus only on large time behaviors of a global classical solution to problem (\ref{PM-ibP}) when the initial brightness $u_0$ has no sharp change in $\Omega$. For this purpose, let $u_0\in C^{2+\alpha}(\bar\Omega)$ be  such that $u_0'(0)=u_0'(1)=0$ and $0<b_0:=\|u_0'\|_\infty<1$. Since $1/b_0>1,$ there exists a unique number $m^*>1$ with $\frac{m^*}{m^*-1}=\frac{1}{b_0}$; that is, $m^*=\frac{1}{1-b_0}$. Let us fix any $0<\tau<1$, $\lambda>0$ and $m>m^*$; whence $b_0\frac{m}{m-1}<1$. Set $\bar{s}=\frac{1}{2}(b_0\frac{m}{m-1}+1)$.

We now modify the Perona-Malik function $\R\ni s\mapsto \frac{s}{1+s^2}$ in order to obtain an odd function $\sigma\in C^3(\R)$ such that
\begin{equation}\label{PM-ibP-1}
\left\{
\begin{array}{l}
  \sigma(s)=\frac{s}{1+s^2}\;\;\forall s\in[-\bar{s},\bar{s}],\\
  \lambda\le \sigma'(s)\le\Lambda\;\;\forall s\in\R, \\
  \mbox{$\sigma'''$ is bounded in $\R$,}
\end{array}\right.
\end{equation}
where $\Lambda>\lambda>0$ are some constants. It now follows from \cite[Theorem 13.24]{Ln} that there exists a unique solution $u\in C^{2,1}(\bar\Omega\times[0,\infty))$ to the initial-boundary value problem
\begin{equation}\label{PM-ibP-2}
\left\{
\begin{array}{ll}
  u_t=(\sigma(u_x))_{x} & \mbox{in $\Omega\times(0,\infty)$,} \\
  u=u_0 & \mbox{on $\Omega\times\{t=0\}$,} \\
  u_x(0,t)=u_x(1,t)=0 & \mbox{for $t\in(0,\infty)$},
\end{array}\right.
\end{equation}
with the property that $u\in C^{2+\alpha,1+\frac{\alpha}{2}}(\bar\Omega\times[0,T])$ for each $T>0$.
Moreover, we can apply \cite[Lemma 2.2]{Ki1} to infer an improved interior regularity of $u$ as $u\in C^{3+\beta,\frac{3+\beta}{2}}(\Omega\times(0,T])$ for all $T>0$, where $\beta\in(0,1)$ is any fixed number. In particular, we have $\rho:=u_x\in C^{1,0}(\bar\Omega\times[0,\infty))\cap C^{2,1}(\Omega\times (0,\infty))$. With $\rho_0:=u_0'\in C^{1}(\bar\Omega)$ and $L=1$, it is then easy to see that $\rho$  solves problem (\ref{main-ibP}); thus both Theorem \ref{thm:uniform-conv} and Lemma \ref{lem:maximum} are applicable.

It follows from Lemma \ref{lem:maximum} that
\[
\min_{\bar\Omega}u_x(\cdot,t_1)\le \min_{\bar\Omega}u_x(\cdot,t_2) \le 0\le\max_{\bar\Omega}u_x(\cdot,t_2)\le \max_{\bar\Omega}u_x(\cdot,t_1)
\]
for all $t_2>t_1\ge 0.$ In particular, we have $\|u_x\|_{L^\infty(\Omega\times(0,\infty))}=\|u_0'\|_\infty=b_0 <b_0\frac{m}{m-1}<\bar{s}<1$. From this, (\ref{PM-ibP-1}) and (\ref{PM-ibP-2}), we see that $u$ is a classical solution to problem (\ref{PM-ibP}).

Next, Theorem \ref{thm:uniform-conv} yields that with $L=1$,
\[
\begin{split}
\|u_x(\cdot,t)\|_\infty\le & \|u_0'\|_\infty \frac{\max\Big\{\frac{\|u_0'\|_\infty \tilde\theta}{(1-\tau)\theta}+1,m \Big\}-e^{-\lambda}}{\max\Big\{\frac{\|u_0'\|_\infty \tilde\theta}{(1-\tau)\theta}+1,m \Big\}-1} \\
&\times \exp\bigg(-\frac{\tau\theta\lambda^2e^{-\lambda}}{\max\Big\{\frac{\|u_0'\|_\infty \tilde\theta}{(1-\tau)\theta}+1,m \Big\}-e^{-\lambda}} t\bigg),
\end{split}
\]
where
\[
\theta:=\min_{[-\|u_0'\|_\infty\frac{m}{m-1}, \|u_0'\|_\infty\frac{m}{m-1}]}\sigma'\;\;\mbox{and}\;\;
\tilde\theta:=\max_{[-\|u_0'\|_\infty\frac{m}{m-1}, \|u_0'\|_\infty\frac{m}{m-1}]}|\sigma''|.
\]
By our choice of the function $\sigma$ in (\ref{PM-ibP-1}), we have
\[
\theta=\frac{1-(b_0\frac{m}{m-1})^2}{((b_0\frac{m}{m-1})^2+1)^2}
\]
and
\[\tilde\theta=\left\{\begin{array}{ll}
                                     \frac{2b_0\frac{m}{m-1}(3-(b_0\frac{m}{m-1})^2)}{((b_0\frac{m}{m-1})^2+1)^3}  & \mbox{if $0<b_0\frac{m}{m-1}<\sqrt{2}-1$,} \\
                                     \frac{3}{4}+\frac{1}{\sqrt{2}} & \mbox{if $\sqrt{2}-1\le b_0\frac{m}{m-1}<1$.}
                                   \end{array}
 \right.
\]

Let us now summarize the above result as follows.

\begin{thm}\label{thm:application-PM}
Let $u_0\in C^{2+\alpha}(\bar\Omega)$ be  such that $u_0'(0)=u_0'(1)=0$ and $\|u_0'\|_\infty<1.$ Then there exists a unique solution $u\in C^{2,1}(\bar\Omega\times[0,\infty))$ to problem (\ref{PM-ibP}) satisfying the following:
\begin{itemize}
\item[(i)] $u\in C^{2+\alpha,1+\frac{\alpha}{2}}(\bar\Omega\times[0,T])$ for each $T>0$,
\item[(ii)] $\int_\Omega u(x,t)\,dx=\int_\Omega u_0(x)\,dx$ for all $t\ge0$,
\item[(iii)] $\min_{\bar\Omega} u(\cdot,t_1)\le \min_{\bar\Omega} u(\cdot,t_2)\le \max_{\bar\Omega} u(\cdot,t_2) \le \max_{\bar\Omega} u(\cdot,t_1)$ for all $t_2>t_1\ge0,$
\item[(iv)] $\min_{\bar\Omega} u_x(\cdot,t_1)\le \min_{\bar\Omega} u_x(\cdot,t_2)\le0\le \max_{\bar\Omega} u_x(\cdot,t_2) \le \max_{\bar\Omega} u_x(\cdot,t_1)$ for all $t_2>t_1\ge 0,$ and
\item[(v)] $\|u(\cdot,t)-\bar{u}_0\|_{W^{1,\infty}(\Omega)} \le Ce^{-\gamma t}$ for all $t\ge0$, where $\bar{u}_0:=\int_0^1 u_0(x)\,dx$, $0<\tau<1$, $\lambda>0$, $m>\frac{1}{1-\|u_0'\|_\infty}>1$,
    \[
\theta:=\frac{1-(\|u_0'\|_\infty\frac{m}{m-1})^2} {((\|u_0'\|_\infty\frac{m}{m-1})^2+1)^2},
\]
\[\tilde\theta:=\left\{\begin{array}{ll}
                                     \frac{2\|u_0'\|_\infty\frac{m}{m-1}(3-(\|u_0'\|_\infty\frac{m}{m-1})^2)}{((\|u_0'\|_\infty\frac{m}{m-1})^2+1)^3}  & \mbox{if $0<\|u_0'\|_\infty\frac{m}{m-1}<\sqrt{2}-1$,} \\
                                     \frac{3}{4}+\frac{1}{\sqrt{2}} & \mbox{if $\sqrt{2}-1\le \|u_0'\|_\infty\frac{m}{m-1}<1$,}
                                   \end{array}
 \right.
\]
\[
\gamma:=\frac{\tau\theta\lambda^2e^{-\lambda}}{\max\Big\{\frac{\|u_0'\|_\infty \tilde\theta}{(1-\tau)\theta}+1,m \Big\}-e^{-\lambda}},
\]
and $C>0$ is a constant depending only on $\|u_0'\|_\infty$, $\sigma$, $\tau$, $\lambda$ and $m$.
\end{itemize}
\end{thm}

In the exponential rate $\gamma$, one has the freedom of choosing any constants $0<\tau<1$, $\lambda>0$ and $m>\frac{1}{1-\|u_0'\|_\infty}$ to fix it. We do not pursue here finding the least upper bound for such rates $\gamma$ only in terms of $\|u_0'\|_\infty$.

\begin{proof}[Proof of Theorem \ref{thm:application-PM}]
The proof is almost complete above. We just mention the unfinished parts.

Item (ii) is an easy consequence of the Neumann boundary condition.

Item (iii) follows from \cite[Proposition 2.4]{KY2}.

Item (v) is a combination of the above result and Poincar\'e's inequality.
\end{proof}

\begin{remk}
Let us interpret the result of Theorem \ref{thm:application-PM} in a viewpoint of image processing. We only consider a slightly noisy 1-D picture of grey level $u_0$ with $\|u_0'\|_\infty<1$, which may be the less interesting case. The Perona-Malik scheme (\ref{PM-ibP}) then essentially does the job of diffusing the image in such a way that
\begin{itemize}
\item the total brightness of the initial image is preserved at all times,
\item the grey level of the present image can go neither above the maximum level nor below the minimum level of the past image,
\item the rate of change in the grey level of the present image cannot be sharper than that of the past image, and
\item the grey level of the initial image uniformly smoothes out to the constant level of the initial mean brightness as $t\to\infty$ at an exponential rate.
\end{itemize}
\end{remk}

\end{document}